\documentclass [12pt]{article}
\usepackage{amssymb,amsmath,amsthm,amscd}
\newtheorem{theorem}{Theorem}
\newtheorem{lemma}{Lemma}
\newtheorem{proposition}{Proposition}

\newtheoremstyle{neosn}{0.5\topsep}{0.5\topsep}{\rm}{}{\sc}{.}{ }{\thmname{#1}\thmnumber{ #2}\thmnote{ {\mdseries#3}}}
\theoremstyle{neosn}
\newtheorem{definition}{Definition}

\newcommand{\Rad}{\,\mathrm{Rad}\,}
\newcommand{\ad}{\,\mathrm{ad}\,}
\newcommand{\kerr}{\,\mathrm{ker}\,}
\newcommand{\GL}{\,\mathrm{GL}\,}
\newcommand{\SL}{\,\mathrm{SL}\,}
\newcommand{\PSL}{\,\mathrm{PSL}\,}

\newcommand{\Sp}{\,\mathrm{Sp}\,}
\newcommand{\Spin}{\,\mathrm{Spin}\,}

\newcommand{\SO}{\,\mathrm{SO}\,}
\newcommand{\Aut}{\,\mathrm{Aut}\,}

\newcommand{\Hom}{\,\mathrm{Hom}\,}
\newcommand{\diag}{\,\mathrm{diag}\,}
\begin{document}
\begin{center}

{\Large {\bf Automorphisms of Chevalley groups\\

\bigskip

of different types over commutative rings }}

\bigskip
\bigskip

{\large \bf E.~I.~Bunina}

\end{center}
\bigskip

\begin{center}

{\bf Abstract.}

\end{center}

In this paper we prove that every automorphism of (elementary)
adjoint
Chevalley group with root system of rank $>1$ over a commutative ring  (with $1/2$ for the systems $A_2$, $F_4$, $B_l$, $C_l$; with $1/2$ and $1/3$ for the system~$G_2$) is standard, i.\,e., it is a composition of ring, inner, central and graph automorphisms.

\bigskip

\section*{Introduction}\leavevmode

Study of automorphism of classical groups was started by the work of  Schreier and van der Varden~\cite{1928} in 1928.
They described all automorphisms of the group $\PSL_n$ $(n\geqslant 3)$ over an arbitrary field.

Diedonne \cite{D} (1951) and Rickart \cite{R1950} (1950) introduced the involution method, and with the help of this method described
automorphisms of the group $\GL_n$ ($n\geqslant 3)$ over a skew field.

The first step in construction the automorphism theory over rings, namely, for the group $\GL_n$ ($n\geqslant 3$) over the ring
of integer numbers, made Hua and Reiner ~\cite{HR}  (1951), after them some papers on commutative integral domains appeared.

The methods of the papers mentioned above were based mostly on studying involutions in the corresponding linear groups.

O'Meara~\cite{O'M2} in 1976 invited very different (geometrical) method, which did not use involutions, with the help  of
this method he described automorphism of the group $\GL_n$ ($n\geqslant 3$) over domains.

In 1982 Petechuk~\cite{v12} described automorphisms of the groups $\GL$, $\SL$ ($n\geqslant 4$)
over arbitrary commutative rings. If $n=3$, then automorphisms of given linear groups are not always standard.
They are standard either  if in a ring $2$ is invertible, or if a ring is a domain, or it is a semisimple ring.

Isomorphisms of the groups $\GL_n(R)$ and $\GL_m(S)$ over arbitrary associative rings with $1/2$
for $n,m\geqslant 3$ were described in 1981 by I.Z.\,Golubchik and A.V.\,Mikhalev~\cite{GolMikh1} and independently by
E.I.\,Zelmanov~\cite{v11}. In 1997 I.Z.\,Golubchik described isomorphisms between these groups for $n,m\geqslant 4$,
but over arbitrary associative rings with~$1$~\cite{Golub}.

In  50-th  years of the previous century Chevalley, Steinberg and others introduced the concept of Chevalley groups
over commutative rings, which includes classical linear groups (special linear $\SL$, special orthogonal $\SO$, symplectic $\Sp$,
spinor $\Spin$, and also projective groups connected with them) over commutative rings.

Clear that isomorphisms and automorphisms of Chevalley groups were also studied intensively.

The description of isomorphisms of Chevalley groups over fields was obtained by
R.\,Steinberg~\cite{Stb1} for the finite case and by J.\,Humphreys~\cite{H} for the infinite one. Many papers were devoted
to description of automorphisms of Chevalley groups over different
commutative rings, we can mention here the papers of
Borel--Tits~\cite{v22}, Carter--Chen~Yu~\cite{v24},
Chen~Yu~\cite{v25}--\cite{v29}, E.\,Abe~\cite{Abe_OSN}, A.\,Klyachko~\cite{Klyachko}.

But the question of description of automorphisms of Chevalley groups over arbitrary commutative rings have still been open.

In the paper  \cite{ravnyekorni} of the author it was shown that automorphisms of adjoint elementary Chevalley groups
with root systems $A_l,D_l, E_l$, $l\geqslant 2$, over local rings with invertible $2$
can be represented as the composition of ring automorphism and an \emph{automorphism--conjugation}, where as automorphism--conjugation
we call a conjugation of elements of a Chevalley group in the adjoint representation by some matrix from
the normalizer of this group in  $\GL(V)$. In the paper \cite{normalizers} according to the results of \cite{ravnyekorni}
it was proved that every automorphism of an arbitrary (elementary) Chevalley group of the described type is standard, i.\,e.,
it is represented as the composition of ring, inner, central and graph automorphism.
In the same paper it was obtained the theorem describing the normalizer of Chevalley groups in their adjoint representation,
which also holds for local rings without $1/2$.

In the papers  \cite{bunF4}, \cite{korni2}, \cite{BunBl} by the same methods we show that all automorphisms of Chevalley groups
with the root systems $F_4$, $G_2$, $B_l, l\geqslant 2$, over local rings with $1/2$ (in the case $G_2$ also with $1/3$)
are standard. In the paper \cite{without2} we described automorphisms of Chevalley groups of types $A_l,D_l, E_l$, $l\geqslant 3$,
over local rings without $1/2$.

In the present paper with the help of results of author's papers
 \cite{ravnyekorni}, \cite{normalizers}, \cite{bunF4}, \cite{korni2}, \cite{BunBl}, \cite{without2},
 and also the methods, described by V.M.\,Petechuk in~\cite{v12} for the special linear group $\SL$,
 we describe automorphisms of adjoint Chevalley groups over arbitrary commutative rings with the assumption that
 the corresponding root systems have rank $>1$, for the root systems $A_2$, $F_4$, $B_l$, $C_l$ the ring contains $1/2$,
 for the system $G_2$ the ring contains $1/2$ and $1/3$.

The author is thankful to N.A.\,Vavilov,  A.A.\,Klyachko,
A.V.\,Mikhalev for valuable advices, remarks and discussions.

\section{Definitions and main theorem.}\leavevmode

We fix a root system~$\Phi$ of rank $>1$. All details about root systems and their properties can be found in
\cite{Hamfris}, \cite{Burbaki}. Suppose now that we have some semisimple complex Lie algebra~$\mathcal L$ of type $\Phi$
with Cartan subalgebra~$\mathcal H$ (detailed information about semisimple Lie algebras can be found in the book~\cite{Hamfris}).

Then we can choose a basis $\{ h_1, \dots, h_l\}$ in~$\mathcal H$ and for every
$\alpha\in \Phi$ elements $x_\alpha \in {\mathcal L}_\alpha$ so that $\{ h_i; x_\alpha\}$ form a basis in~$\mathcal L$ and for every two elements of this basis their commutator is an integral linear combination of the elements of the same basis.

Let us introduce elementary Chevalley groups (see,
for example,~\cite{Steinberg}).

Let  $\mathcal L$ be a semisimple Lie algebra (over~$\mathbb C$)
with a root system~$\Phi$, $\pi: {\mathcal L}\to gl(V)$ be its
finitely dimensional faithful representation  (of dimension~$n$). If
$\mathcal H$ is a Cartan subalgebra of~$\mathcal L$, then a
functional
 $\lambda \in {\mathcal H}^*$ is called a
 \emph{weight} of  a given representation, if there exists a nonzero vector $v\in V$
 (that is called a  \emph{weight vector}) such that
for any $h\in {\mathcal H}$ $\pi(h) v=\lambda (h)v.$

In the space~$V$ there exists a basis of weight vectors such that
all operators $\pi(x_\alpha)^k/k!$ for $k\in \mathbb N$ are written
as integral (nilpotent) matrices. This basis is called a
\emph{Chevalley basis}. An integral matrix also can be considered as
a matrix over an arbitrary commutative ring with~$1$. Let $R$ be
such a ring. Consider matrices $n\times n$ over~$R$, matrices
$\pi(x_\alpha)^k/k!$ for
 $\alpha\in \Phi$, $k\in \mathbb N$ are included in $M_n(R)$.

Now consider automorphisms of the free module $R^n$ of the form
$$
\exp (tx_\alpha)=x_\alpha(t)=1+tx_\alpha+t^2 (x_\alpha)^2/2+\dots+
t^k (x_\alpha)^k/k!+\dots
$$
Since all matrices $x_\alpha$ are nilpotent, we have that this
series is finite. Automorphisms $x_\alpha(t)$ are called
\emph{elementary root elements}. The subgroup in $Aut(R^n)$,
generated by all $x_\alpha(t)$, $\alpha\in \Phi$, $t\in R$, is
called an \emph{elementary adjoint Chevalley group} (notation:
$E_{\ad}(\Phi,R)=E_{\ad}(R)$).

The action of $x_\alpha(t)$ on the Chevalley basis is described in
\cite{v23}, \cite{VavPlotk1}.

All weights of a given representation (by addition) generate a
lattice (free Abelian group, where every  $\mathbb Z$-basis  is also
a $\mathbb C$-basis in~${\mathcal H}^*$), that is called the
\emph{weight lattice} $\Lambda_\pi$.

Elementary Chevalley groups are defined not even by a representation
of the Chevalley groups, but just by its \emph{weight lattice}.
Namely, up to an abstract isomorphism an elementary Chevalley group
is completely defined by a root system~$\Phi$, a commutative
ring~$R$ with~$1$ and a weight lattice~$\Lambda_\pi$.

Among all lattices we can mark  the lattice corresponding to the
adjoint representation: it is generated by all roots (the \emph{root
lattice}~$\Lambda_{ad}$). The corresponding (elementary) Chevalley group is called \emph{adjoint}.

Introduce now Chevalley groups (see~\cite{Steinberg},
\cite{Chevalley}, \cite{v3}, \cite{v23}, \cite{v30}, \cite{v43},
\cite{VavPlotk1}, and also latter references in these papers).

Consider semisimple linear algebraic groups over algebraically
closed fields. These are precisely elementary Chevalley groups
$E_\pi(\Phi,K)$ (see.~\cite{Steinberg},~\S\,5).

All these groups are defined in $SL_n(K)$ as  common set of zeros of
polynomials of matrix entries $a_{ij}$ with integer coefficients
 (for example,
in the case of the root system $C_l$ and the universal
representation we have $n=2l$ and the polynomials from the condition
$(a_{ij})Q(a_{ji})-Q=0$). It is clear now that multiplication and
taking inverse element are also defined by polynomials with integer
coefficients. Therefore, these polynomials can be considered as
polynomials over arbitrary commutative ring with a unit. Let some
elementary Chevalley group $E$ over~$\mathbb C$ be defined in
$SL_n(\mathbb C)$ by polynomials $p_1(a_{ij}),\dots, p_m(a_{ij})$.
For a commutative ring~$R$ with a unit let us consider the group
$$
G(R)=\{ (a_{ij})\in \SL_n(R)\mid \widetilde p_1(a_{ij})=0,\dots
,\widetilde p_m(a_{ij})=0\},
$$
where  $\widetilde p_1(\dots),\dots \widetilde p_m(\dots)$ are
polynomials having the same coefficients as
$p_1(\dots),\dots,p_m(\dots)$, but considered over~$R$.

This group is called the \emph{Chevalley group} $G_\pi(\Phi,R)$ of
the type~$\Phi$ over the ring~$R$, and for every algebraically
closed field~$K$ it coincides with the elementary Chevalley group.

The subgroup of diagonal (in the standard basis of weight vectors)
matrices of the Chevalley group $G_\pi(\Phi,R)$ is called the
 \emph{standard maximal torus}
of $G_\pi(\Phi,R)$ and it is denoted by $T_\pi(\Phi,R)$. This group
is isomorphic to $\Hom(\Lambda_\pi, R^*)$.

Let us denote by $h(\chi)$ the elements of the torus $T_\pi
(\Phi,R)$, corresponding to the homomorphism $\chi\in Hom
(\Lambda(\pi),R^*)$.

In particular, $h_\alpha(u)=h(\chi_{\alpha,u})$ ($u\in R^*$, $\alpha
\in \Phi$), where
$$
\chi_{\alpha,u}: \lambda\mapsto u^{\langle
\lambda,\alpha\rangle}\quad (\lambda\in \Lambda_\pi).
$$

Note that the condition
$$
G_\pi (\Phi,R)=E_\pi (\Phi,R)
$$
is not true even for fields, that are not algebraically closed.

Let us show the difference between Chevalley groups and their elementary subgroups in the case when $R$ is semilocal. In this case
 $G_\pi (\Phi,R)=E_\pi(\Phi,R)T_\pi(\Phi,R)$
(see~\cite{Abe1}), and elements $h(\chi)$ are connected with elementary generators
by the formula
\begin{equation}\label{ee4}
h(\chi)x_\beta (\xi)h(\chi)^{-1}=x_\beta (\chi(\beta)\xi).
\end{equation}


Define four types of automorphisms of a Chevalley group
 $G_\pi(\Phi,R)$, we
call them  \emph{standard}.

{\bf Central automorphisms.} Let $C_G(R)$ be a center of
$G_\pi(\Phi,R)$, $\tau: G_\pi(\Phi,R) \to C_G(R)$ be some
homomorphism of groups. Then the mapping $x\mapsto \tau(x)x$ from
$G_\pi(\Phi,R)$ onto itself is an automorphism of $G_\pi(\Phi,R)$,
that is denoted by~$\tau$ and called a \emph{central automorphism}
of the group~$G_\pi(\Phi,R)$.

{\bf Ring automorphisms.} Let $\rho: R\to R$ be an automorphism of
the ring~$R$. The mapping $(a_{i,j})\mapsto (\rho (a_{i,j}))$ from $G_\pi(\Phi,R)$
onto itself is an automorphism of the group $G_\pi(\Phi,R)$, that is
denoted by the same letter~$\rho$ and is called a \emph{ring
automorphism} of the group~$G_\pi(\Phi,R)$. Note that for all
$\alpha\in \Phi$ and $t\in R$ an element $x_\alpha(t)$ is mapped to
$x_\alpha(\rho(t))$.

{\bf Inner automorphisms.} Let $S$ be some ring containing~$R$,  $g$
be an element of $G_\pi(\Phi,S)$, that normalizes the subgroup $G_\pi(\Phi,R)$. Then
the mapping $x\mapsto gxg^{-1}$  is an automorphism
of the group~$G_\pi(\Phi,R)$, that is denoted by $i_g$ and is called an
\emph{inner automorphism}, \emph{induced by the element}~$g\in G_\pi(\Phi,S)$. If $g\in G_\pi(\Phi,R)$, then call $i_g$ a \emph{strictly inner}
automorphism.

{\bf Graph automorphisms.} Let $\delta$ be an automorphism of the
root system~$\Phi$ such that $\delta \Delta=\Delta$. Then there
exists a unique automorphisms of $G_\pi (\Phi,R)$ (we denote it by
the same letter~$\delta$) such that for every $\alpha \in \Phi$ and
$t\in R$ an element $x_\alpha (t)$ is mapped to
$x_{\delta(\alpha)}(\varepsilon(\alpha)t)$, where
$\varepsilon(\alpha)=\pm 1$ for all $\alpha \in \Phi$ and
$\varepsilon(\alpha)=1$ for all $\alpha\in \Delta$.

Now suppose that $\delta_1,\dots, \delta_k$ are all different graph automorphisms for the given root system  (for the systems $E_7,E_8,B_l,C_l,F_4,G_2$ there can be just identical automorphism, for the systems $A_l, D_l,  l\ne 4, E_6$ there are two such
automorphisms, for the system $D_4$ there are six automorphisms). Suppose that we have a system of orthogonal idempotents of
the ring~$R$:
$$
\{\varepsilon_1, \dots, \varepsilon_k\mid \varepsilon_1+\dots+\varepsilon_k=1, \forall i\ne j\ \varepsilon_i\varepsilon_j=0\}.
$$
Then the mapping
$$
\Lambda_{\varepsilon_1,\dots,\varepsilon_k}:= \varepsilon_1 \delta_1+\dots+ \varepsilon_k \delta_k
$$
of the Chevalley group onto itself is an  automorphism, that is called a \emph{graph automorphism} of the Chevalley
group $G_\pi (\Phi,R)$.

 Similarly we can define four type of automorphisms of the elementary
subgroup~$E(R)$. An automorphism~$\sigma$ of the group
 $G_\pi(\Phi,R)$ (or $E_\pi(\Phi,R)$)
is called  \emph{standard} if it is a composition of automorphisms
of these introduced four types.

Our aim is to prove the next main theorem:

\begin{theorem}\label{main}
Let $G=G_{\pi}(\Phi,R)$ $(E_\pi(\Phi,R))$
be an \emph{(}elementary\emph{)} adjoint Chevalley group of rank $>1$  $R$ be a commutative ring with~$1$. Suppose that for $\Phi = A_2, B_l, C_l$ or $F_4$ we have $1/2\in R$, for $\Phi=G_2$ we have $1/2,1/3 \in R$. Then every automorphism of the group~$G$
is standard and the inner automorphism in the composition is strictly inner.
\end{theorem}

\section{Known notions, definitions and results, which will be used in the proof}\leavevmode

\subsection{Localization of rings and modules; injection of a ring into the product of its localizations.}\leavevmode

\begin{definition}  Let $A$ be a commutative ring. A subset $S\subset A$ is called \emph{multiplicatively closed} in~$A$, if $1\in S$ and $S$ is closed under multiplication.
\end{definition}

Introduce  an equivalence relation $\sim$ on the set of pairs $A\times S$ as follows:
$$
\frac{a}{s}\sim \frac{b}{t} \Longleftrightarrow \exists u\in S:\ (at-bs)u=0.
$$
  By $\frac{a}{s}$ we denote the whole equivalence class of the pair $(a,s)$, by $S^{-1}R$ we denote the set of all equivalence classes. On the set $S^{-1}R$ we can introduce the ring structure by
$$
\frac{a}{s}+\frac{b}{t}=\frac{at+bs}{st},\quad \frac{a}{s}\cdot \frac{b}{t}=\frac{ab}{st}.
$$

\begin{definition}
The ring $S^{-1}A$ is called the \emph{ring of fractions of~$A$ with respect to~$S$}.
\end{definition}

 Let $\mathfrak p$ be a prime ideal of~$A$. Then the set $S=A\setminus {\mathfrak p}$ is multiplicatively closed (it is equivalent to the definition of the prime ideal). We will denote the ring of fractions  $S^{-1}A$ in this case by $A_{\mathfrak p}$. The elements $\frac{a}{s}$, $a\in \mathfrak p$, form an ideal $\mathfrak M$ in~$A_{\mathfrak p}$. If $\frac{b}{t}\notin \mathfrak M$, then $b\in S$, therefore $\frac{b}{t}$ is invertible in~$A_{\mathfrak p}$. Consequently the ideal $\mathfrak M$ consists of all non-invertible elements of the ring~$A_{\mathfrak p}$, i.\,e., $\mathfrak M$ is the greatest ideal of this ring, so $A_{\mathfrak p}$ is a local ring.

The process of passing from~$A$ to~$A_{\mathfrak p}$ is called  \emph{localization at~${\mathfrak p}$.}

The construction $S^{-1}A$ can be easily carried trough  with an  $A$-module~$M$.
 Let $m/s$ denote the equivalence class of the pair $(m,s)$, the set $S^{-1}M$ of all such fractions    is made as a module $S^{-1}M$ with obvious operations of addition and scalar multiplication. As above we will write $M_{\mathfrak p}$ instead of $S^{-1}M$ for $S=A\setminus {\mathfrak p}$, where $\mathfrak p$ is a prime ideal of~$A$.

\begin{proposition}\label{inlocal}
Every commutative ring  $A$ with $1$ can be naturally embedded in the cartesian product of all its localizations  by maximal ideals
 $$
S=\prod\limits_{{\mathfrak m}\text{ is a maximal ideal of }A} A_{\mathfrak m}
$$
by diagonal mapping, which corresponds every $a\in A$ to the element
$$
\prod\limits_{\mathfrak m} \left( \frac{a}{1}\right)_{\mathfrak m}
$$
of~$S$.
\end{proposition}

\subsection{Isomorphisms of Chevalley groups over fields.}\leavevmode

We will need the description of isomorphisms between Chevalley groups over fields.

Suppose that root systems under consideration have ranks $>1$.

Introduce an additional concept of \emph{diagonal} automorphism:

\begin{definition}(see Lemma~58 from the book~\cite{Steinberg}).
 Let $G$ be a (elementary) Chevalley group over a field~$k$, and suppose that we have some set of elements $f_\alpha \in k^*$
 for all simple roots~$\alpha\in \Phi$. Let us extend  $f$ to an homomorphism of the whole lattice, generated by all roots, into~$k^*$. Then there exists a unique automorphism $\varphi$ of the group~$G$ such that
 $$
\varphi(x_\alpha(t))= x_\alpha (f_\alpha t)\text{ для всех }\alpha\in \Phi, t\in k.
$$
This automorphism is called a \emph{diagonal automorphism}.
\end{definition}

This is the description of isomorphisms of Chevalley groups over fields:

\begin{theorem}[see Theorems~30  and~31 from \cite{Steinberg}]\label{isom_fields}
 Let $G$, $G'$ be  \emph{(}elementary\emph{)} Chevalle groups, constructed  with root systems $\Phi, \Phi'$ and fields $k,k'$, respectively. Suppose that the root systems are not decomposable and have ranks $>1$. Suppose that for the root systems  $B_l,C_l,F_4$ corresponding fields have characteristics $\ne 2$ and for the root system  $G_2$ it is not equal to three. Let $\varphi: G\to G'$ be a group isomorphism. Then the root systems $\Phi$ and $\Phi'$ coincide, the fields $k$ and $k'$ are isomorphic, and the isomorphism $\varphi$ is a composition of a ring isomorphism between $G$ and $G'$, and also inner, diagonal and graph automorphisms of the group~$G'$. If the groups  $G$ and $G'$ are adjoint, then there is no diagonal automorphism in the composition.
\end{theorem}

\subsection{Normal structure of Chevalley groups over commutative rings.}\leavevmode

Note that for every ideal $I$ of~$R$ the natural mapping $R\to R/I$ induces a homomorphism
$$
\lambda_I: G_\pi(\Phi,R)\to G_\pi(\Phi,R/I).
$$
If $I$ is a proper ideal of~$R$, then the kernel of~$\lambda_I$  is a  non-central normal subgroup of~$G_\pi(\Phi, R)$.

We denote the inverse image of the center of $G_\pi(R/I)$ under~$\lambda_I$ by~$Z_\pi(\Phi,R,I)$.

\begin{theorem}\label{t3_1}\emph{(see \cite{v19})}
Let the rank of an indecomposable root system~$\Phi$ is more than one. If a subgroup $\mathcal H$ of~$E_\pi(\Phi,R)$
is normal in~$E_\pi(\Phi,R)$, then
$$
E_\pi(\Phi,R,I)\leqslant {\mathcal H}\leqslant Z_\pi(\Phi,R,I)\cap E_\pi (\Phi,R)
$$
for some uniquely defined ideal~$I$ of the ring~$R$.
\end{theorem}

\subsection{Projective modules over local rings.}\leavevmode

The well-known result is the following

\begin{theorem}\label{projfree}
A finitely generated projective module over a local ring is free.
 \end{theorem}

\subsection{The subgroup  $E_\pi(\Phi,R)$ is characteristic in the group $G_\pi(\Phi,R)$.}\leavevmode

A subgroup  $H$ of $G$ is called \emph{characteristic}, if it is mapped into itself under any automorphism of~$G$.
In particular, any characteristic subgroup is normal.

\begin{theorem}\emph{(see \cite{v42})}.\label{character}
If the rank of~$\Phi$ is greater than one, the elementary subgroup $E_\pi(\Phi,R)$ is characteristic in the Chevalley group $G_\pi(\Phi,R)$.
\end{theorem}

\section{Formulation of main steps of the proof.}\leavevmode

 If $R$ is a ring, $I$  is its ideal, then by $\lambda_I: G_\pi(\Phi,R)\to G_\pi(\Phi,R/I)$ ($E_\pi(\Phi,R)\to E_\pi(\Phi,R/I)$) we denote the homomorphism which corresponds every element (matrix) $A\in G_\pi(\Phi,R)$ to its image under the natural homomorphism $R\to R/I$.

Recall that by $Z_I$ we denote the inverse image of the center of the group $G_\pi(\Phi,R/I)$ under the homomorphism $\lambda_I$.

 \begin{definition}
 Let $C_I$ denote the group $Z_I\cap E_\pi(\Phi,R)$,  $N_I=\kerr \lambda_I\cap E_\pi(\Phi,R)$.
 \end{definition}

\begin{proposition}\label{p4_1}
 Let $\varphi$ be an arbitrary automorphism of the group $E_\pi(\Phi,R)$, $I$ be a maximal ideal of~$R$. Then there exists a maximal ideal $J$ of~$R$ such that $\varphi (N_I)=N_J$.
\end{proposition}

\begin{proof}
It is clear that the group $C_I$ is normal in $E_\pi(\Phi,R)$. As it follows from  Theorem~\ref{t3_1}, for such a subgroup $G$  we have an inclusion
$$
E_I\subseteq G\subseteq  C_I,
$$
therefore the subgroups $C_I$, and only they are maximal normal subgroups of the group~$E_\pi(\Phi,R)$. Consequently, for a maximal ideal $I$ of the ring~$R$ there exists a maximal ideal $J$ of~$R$ such that $\varphi(C_I)=C_J$. Show that $\varphi(N_I)=N_J$.

Consider the group $G=E_\pi(\Phi,R)/C_I=E_\pi(\Phi,R)/ (Z_I\cap E_\pi(\Phi,R))$. It is isomorphic to $E_\pi(\Phi,R)\cdot Z_I/Z_I$. Now use the Isomorphism theorem, namely, let us factorize the both parts by~$C_I$.

As result we obtain
$E_\pi(\Phi,R/I) \cdot Z(G_\pi(\Phi,R/I))/ Z(G_\pi(\Phi,R/I))$. It is isomorphic to $G\cong E_\pi(\Phi,R/I)/(E_\pi(\Phi,R/I)\cap Z(G_\pi(\Phi,R/I))\cong E_{\ad}(\Phi,R/I)$. Therefore $E_\pi(\Phi,R)/C_I\cong E_{\ad}(\Phi,R/I)$.

 Since $\varphi(C_I)=C_J$, then the automorphism $\varphi$ induces an isomorphism $\overline \varphi$ of the groups $E_\pi(\Phi,R)/C_I\cong E_{\ad}(\Phi,R/I)$ and $E_\pi(\Phi,R)/C_J\cong E_{\ad}(\Phi,R/J)$ such that the diagramm
$$
\begin{CD}
E_\pi(\Phi,R) @> \varphi >> E_\pi(\Phi,R)\\
@VVV @VVV\\
E_{\ad}(\Phi,R/I)   @> \overline \varphi >> E_{\ad}(\Phi,R/J)
\end{CD}
$$
is commutative. Isomorphisms of the groups $E_{\ad}$ with root systems under consideration we have described in Theorem~\ref{isom_fields}. So we see that the fields $R/I$ and $R/J$ are  isomorphic (we denote the corresponding isomorphism by~$\rho$) and  $\overline \varphi( A) = i_{\overline g} \overline  \delta (\overline \rho (A))$ for every  $A\in E_{\ad}(\Phi,R/I)$, $\overline g\in G_{\ad}(\Phi,R/J)$,  $\delta$ is a graph automorphism of $G_{\pi}(\Phi,R/J)$.

Since a graph automorphism of the group $G_\pi(\Phi,R/J)$ can be expanded to a graph automorphism of the group   $E_\pi(\Phi,R)$, and the last one maps the group $N_J$ into itself, it is sufficient to consider the case, when the graph automorphism in the composition is identical.

We obtain that in the group $E_\pi(\Phi,R)$ there is the equality
$$
 \lambda_J\varphi (x_\alpha(t))=g (x_\alpha(\rho(t+I))) g^{-1} c,\quad
 c\in Z(E_\pi(\Phi,R/J)).
 $$
  Since $x_\alpha(t)$ is always (for the root systems under consideration) a product of commutators of elements $x_\beta(s)$, then the central element~$c$ disappears from the image. Therefore we have
$$
 \lambda_J\varphi (x_\alpha(t))=g (x_\alpha(\rho(t+I)) g^{-1}.
 $$

Let now $M=x_{\alpha_1}(t_1)\dots x_{\alpha_k}(t_k)$ be an arbitrary element of~$N_I$. Then
 $$
 \lambda_J \varphi (M)=g (x_{\alpha_1}(\rho (t_1+I))\dots x_{\alpha_k}(\rho(t_k+I)))g^{-1}=g (\overline \rho \lambda_I(M)) g^{-1}=E.
 $$
   Consequently $\varphi  (N_I)\subseteq  N_J$. Clear that the inclusion   $\varphi^{-1} (N_J)\subseteq N_I$ is proved similarly. So
   $\varphi(N_I)=N_J$.
   \end{proof}

   Consider a ring $R$ and its maximal ideal~$I$. We denote the localization $R$ with~$I$ by~$R_I$ again, and its radical (the greatest ideal) we denote by $\Rad R_I$. Note that we have to isomorphic fields  $R/I$ and $R_I/\Rad R_I$.
   Therefore we  can turn the arrow $\mu_I$ in the diagram
   $$
   \begin{CD}
   R @>>> R_I\\
   @V\lambda_I VV @VV \lambda_{\Rad R_I}V\\
   R/I @>\mu_I >> R_I/\Rad R_I
   \end{CD}
   $$

   Let now $\varphi$ be an arbitrary automorphism of $E_\pi(\Phi,R)$. Proposition \ref{p4_1} gives us a possibility to consider the commutative diagram
   \begin{equation}\label{diagramm}
   \begin{CD}
   E_\pi(\Phi,R) @> \varphi >>  E_\pi(\Phi,R)\\
   @V r_I VV @VVr_J V\\
   E_\pi(\Phi,R_I) @. E_\pi(\Phi,R_J)\\
   @V \lambda_{\Rad R_I}VV @VV \lambda_{\Rad R_J}V\\
   E_\pi(\Phi,R_I/\Rad R_I) @. E_\pi(\Phi,R_J /\Rad R_J)\\
   @V s_I VV @VV s_J V \\
   E_\pi (\Phi,R/I) @> \overline \varphi >> E_\pi(\Phi,R/J)
   \end{CD}
   \end{equation}

     The groups $E_\pi(\Phi,R/I)$ and $E_\pi(\Phi,R/J)$ are just elementary Chevalley groups over fields, their isomorphisms we have already described in Theorem~\ref{isom_fields}.

    Recall that the fields $R/I$ and $R/J$ are isomorphic  (as earlier we denote the corresponding isomorphism by~$\rho$), and also
    $$
    \overline \varphi (A) = i_g \circ f\circ \delta_i \rho(A)\ \forall A\in E_\pi(\Phi,R/I),\quad  g\in G_\pi(\Phi, R/J),
    $$
here $\delta_i$ is one of graph automorphisms, $f$ is a diagonal automorphism.

    The description of automorphisms of the group $E_\pi(\Phi,R)$ can be made by the following scheme. The ring $R$ is embedded into the ring $S=\prod R_I$, which is the Cartesian product of all local rings~$R_I$, obtained by localization  the ring~$R$ with different maximal ideals~$I$. We denote by $R_i$ the ring $\prod R_I$, where maximal ideals are taken such that in the composition we have namely the graph automorphism  $\delta_i$. Clear that $S=R_1\oplus \dots \oplus R_k$. Let in the ring~$S$ $a_i=(0,\dots,0,1,0,\dots,0)$.

    Clear that the group $E_\pi(\Phi,R)$ is embedded in
        $$
    G_\pi(\Phi,S)=G_\pi(\Phi,\prod R_I)=G_\pi(\Phi,R_1\oplus  \dots \oplus  R_k)=G_\pi(\Phi,R_1)\times \dots \times G_\pi(\Phi,R_k).
    $$

    {\bf The first step.}
    We prove that for every maximal ideal~$J$
    $$
    r_J  \varphi (x_\alpha(1))=i_{g_J} \delta_i r_J(x_\alpha(1)) ,
    $$
    where $g_J\in G_\pi(\Phi,\overline{R_J})$ (an extension of the ring $R_J$), $i$ is such that  $R_J\in R_i$.

    {\bf The second step.}

 We consider adjoint Chevalley groups.

We show that actually the idempotents $a_i$ belong to the ring~$R$, and the inner automorphism of $G_\pi(\Phi,S)$, generated by $g=\prod g_J$, induced an automorphism of the group~$G_\pi(\Phi,R)$.

Then we show that if we take the composition of the initial automorphism, the inner automorphism  $i_{g^{-1}}$ and the graph automorphism $\Lambda_{a_1,\dots,a_k}$, then the obtained automorphism is ring.

\medskip

Now suppose that the both steps are proved. Then we have the description of automorphisms of the elementary subgroup
 $E_\pi(\Phi,R)$, and also we know that in the composition there is no central automorphism.

If we have now some automorphism of the group~$G_{\ad}(\Phi,R)$, then it induces an automorphism of the group $E_{\ad}(\Phi,R)$ (see Theorem~\ref{character}), which is standard (the composition of ring, inner and graph automorphisms). All these three automorphisms of the group $E_{\ad}(\Phi,R)$ are extended to the automorphisms of the group~$G_{\ad}(\Phi,R)$. Therefore multiplying   $\varphi$ to the suitable
standard automorphism, we can assume that  $\varphi$ is identical on the subgroup $E_{\ad}(\Phi,R)$.
As above it means
$$
\forall A\in E_{\ad}(\Phi,R) \forall g\in G_{\ad}(\Phi,R) gAg^{-1}=\varphi(gAg^{-1})=\varphi(g) A \varphi(g)^{-1},
$$
Consequently $\varphi(g)g^{-1}$ commutes with all $A\in E_{\ad}(\Phi,R)$, i.\,e., it belongs to the center of~$G_{\ad}(\Phi,R)$. Therefore, $\varphi$ is a central automorphism. $\square$

\section{Proof of the first step in the theorem.}\leavevmode

For our convenience we will suppose that a Chevalley group under consideration is adjoint.

A graph automorphism of $G_{\ad}(\Phi, R/J)$ in the diagram~\eqref{diagramm} can be expanded to an automorphism of the whole group~$E_{\ad}(\Phi,R)$. Therefore we can assume the automorphism  $\varphi$ such that $\overline \varphi = i_{\overline g} \circ \overline \delta$ (according to the fact that adjoint Chevalley groups over fields have no diagonal automorphisms).

Consider an arbitrary element $x_\alpha(1)\in E_{\ad} (\Phi, R)$, $\alpha\in \Phi$. Its image  under the mapping $r_I$ is also $x_\alpha (1)=x_\alpha (1/1)\in E_{\ad} (\Phi, R_I)$. In the field $R/I$ its image has the same form. The element $x_\alpha'=\varphi(x_\alpha(1))\in E_{\ad} (\Phi, R)$ being factorized by the ideal~$J$ gives $\overline \varphi(x_\alpha(1))=i_{\overline g}  (x_\alpha (1))$, where $\overline g\in G_{\ad}(\Phi,R/J)$.

Choose now $g\in G_{\ad}(\Phi,R_J)$ such that under factorization $R_J$ by its radical the element $g$ corresponds to~$\overline g$.

Now consider the following mapping $\psi: E_{\ad} (\Phi, R)\to E_{\ad} (\Phi, R_J)$:
$$
\psi = i_{g^{-1}} \circ r_J\circ \varphi.
$$
Under  $\psi$ all $x_\alpha (1)$, $\alpha\in \Phi$, are corresponded to such $x_\alpha'$, that $x_\alpha(1)-x_\alpha'\in M_N(\Rad R_J)$.

Therefore we obtain a set of elements $\{ x_\alpha' \mid \alpha \in \Phi\}\subset E_{\ad} (\Phi,R_J)$, satisfying all the same conditions as $\{ x_\alpha (1)\mid \alpha \in \Phi\}$, and also equivalent to $x_\alpha(1)$ modulo radical of~$R_J$.

It is precisely the situation of papers \cite{normalizers}, \cite{bunF4}, \cite{korni2}, \cite{BunBl}, \cite{without2}, where for a local ring~$S$ and root systems $A_2, B_l, C_l, F_4$ for $2\in S^*$, $G_2$ for $2,3\in S^*$, the root systems $A_l$, $l\geqslant 3$, $D_l$, $E_6,E_7, E_8$ without any additional conditions it was proved that if in the group $E_{\ad} (\Phi, S)$ some elements $x_\alpha'$ are the images of the corresponding  $x_\alpha(1)$, $\alpha\in \Phi$, and also $x_\alpha(1)-x_\alpha'\in M_N(\Rad S)$, then there exists $g'\in G_{\ad} (\Phi, S)$ , $g'-E\in M_N(\Rad S)$,  such that for every $\alpha\in \Phi$
$$
x_\alpha(1)= i_{g'} (x_\alpha').
$$

Therefore the first step of our theorem completely follows from the above statement. $\square$

Embedding now the initial ring $R$ into the ring $S=\prod\limits_J R_J$, we see that
$$
\varphi(x_\alpha(1))=\Lambda_{e_1,\dots,e_k} g (x_\alpha(1)) g^{-1},
$$
where $g=\prod\limits_J g_J$, $e_i$ are idempotents of~$S$,  introduced above.

\section{Proof of the second step in the theorem.}\leavevmode

We know now that the automorphism $\varphi$ satisfies the equality
$$
\varphi(x_\alpha(1))=g x_{\delta_1(\alpha)}(e_1)\dots x_{\delta_k(\alpha)}(e_k) g^{-1}\in E_{\ad}(\Phi,R), \quad g=\prod g_J.
$$

Note that for the root systems $B_l, C_l, E_7, E_8, F_4, G_2$ $k=1$, for the systems $A_l, D_l$ ($l\geqslant 5$), $E_6$ $k=2$,
for the system $D_4$ $k=6$.

As above we assume now that the Chevalley groups $G_\pi(\Phi, R)$ and $G_\pi(\Phi, S)$ are adjoint, i.\,e., $\pi=\ad$.

In this case every graph automorphism of the Chevalley group $G_{\ad}(\Phi,S)$ is realized by some matrix
 $\Lambda=e_1\Lambda_1+\dots +e_k
\Lambda_k\in \GL_N(S)$, the matrices $\Lambda_1,\dots, \Lambda_k$ have integer koefficients.

Therefore the composition of conjugation by $g\in G_{\ad}(\Phi,S)$ and the graph automorphism (denote it by~$\psi$)
can be continued to the whole matrix ring $M_N(S)$.

\begin{lemma}\label{generate}
Under all theorem assumptions the elements $x_\alpha(1)$, $\alpha\in \Phi$, by addition and multiplication
generate the whole basis $X_\alpha$, $\alpha\in \Phi$, of the Lie algebra ${\mathcal L}(\Phi)$.
\end{lemma}
\begin{proof}
If the root system differs from $G_2$ and $1/2\in R$, then $x_\alpha (1)=E+X_\alpha  + X_\alpha^2/2$, therefore
$X_\alpha=x_\alpha(1)-E-(x_\alpha(1)-E)^2/2$.

For the root system $G_2$ and a short root~$\alpha$ we have $x_\alpha(1)=E+X_\alpha+X_\alpha^2/2+X_\alpha^3/6$. We  suppose that
 $1/6\in R$, then $X_\alpha^3/6=(x_\alpha(1)-E)^3/6$, $X_\alpha^2/2=(x_\alpha(1)-E)^2/2-X_\alpha^3$, therefore we easily get
$X_\alpha$.

Suppose now that we deal with systems $A_l$, ($l\geqslant 3$),  $D_l, E_l$, two is not invertible.

In this case $x_\alpha(1)=E+X_\alpha+ X_\alpha^2/2$, where $X_\alpha^2/2=E_{\alpha,-\alpha}$. Choose any two roots
 $\gamma$, $\beta\in \Phi$ so that $\gamma+\beta=\alpha$. Using the condition
$$
(x_\gamma(1)x_\beta(1)-x_\gamma(1)-x_\beta(1)+E)^2=E_{\alpha,-\alpha},
$$
we obtain $X_\alpha$.

The lemma is proved.

\end{proof}

  From Lemma \ref{generate} we see that the automorphism $\psi$ of the matrix ring $M_N(S)$ maps  the matrices $X_\alpha$, $\alpha\in \Phi$,
   into the matrices with coefficients from the ring~$R$. Therefore any matrix from ${\mathcal L}(\Phi,R)$ under the action
   of the conjugation $\psi$ is mapped to a matrix from $M_N(R)$.

   Since $x_\alpha (t)=E+t X_\alpha+ t^2 X_\alpha/2+\dots$ for any $\alpha\in \Phi$, $t\in R$,
   then every matrix $x_\alpha(t)$ under the action of  $\psi$ is mapped to a matrix from $\GL_N(R)$.
   Consequently $\psi(E_{\ad}(\Phi,R))\subset \GL_N(R)$.

  From another side, the conjugation $\psi$ is the composition of inner and graph automorphisms of the Chevalley group
   $G_{\ad}(\Phi,S)$ (and its elementary subgroup $E_{\ad}(\Phi,S)$), so it is an automorphism of the group $E_{\ad}(\Phi,S)$.
   Since the image of the Chevalley group $E_{\ad}(\Phi,R)$ under $\psi$ belongs to the Chevalley group $E_{\ad}(\Phi,S)$
   and also to the ring $M_N(R)$, then $\psi(E_{\ad}(\Phi,R))=E_{\ad}(\Phi,R)$.

  Therefore taking the composition of the initial automorphism $\varphi\in \Aut (E_{\ad}(\Phi,R))$
  and the automorphism $\psi^{-1}\in \Aut(E_{\ad}(\Phi,R))$, we obtain some automorphism
   $\rho = \psi^{-1}\circ \varphi\in \Aut (E_{\ad}(\Phi,R))$ such that $\rho(x_{\alpha}(1))=x_{\alpha}(1)$ for every $\alpha\in \Phi$.

\begin{lemma}\label{ringaut}
Under the initial assumptions of the theorem $\rho$ is a ring automorphism of the Chevalley group.
\end{lemma}
\begin{proof}
At first we suppose that in the ring~$R$ the element $2$ is invertible (for the root system~$G_2$ also $1/3\in R$).

Our first step is to prove lemma for the root system $A_2$. In the system $A_2$ there are six roots:
 $\pm \alpha_1, \pm \alpha_2 \pm \alpha_3=\pm (\alpha_1+\alpha_2)$
 (detailed matrices for the root system $A_2$ can be found in the paper~\cite{ravnyekorni}).

Let $\rho (x_{\alpha_1}(t))=y$.
Note that $y$ commutes with
$$
h_{\alpha_1}(-1)=\diag[1,1,-1,-1,-1,-1,1,1],
$$
 therefore we get that
$y$ is block-diagonal up to the basis parts $\{\alpha_1,-\alpha_1, h_1,h_2\}$ and
$\{\alpha_2,-\alpha_2,\alpha_1+\alpha_2, -\alpha_1-\alpha_2\}$. Then $y$ commutes with
 $x_{\alpha_1}(1)$, $x_{-\alpha_2}(1)$ and $x_{\alpha_1+\alpha_2}(1)$ so by direct calculations we obtain
$$
y=\begin{pmatrix}
y_{1,1}& y_{1,2}&  0& 0& 0& 0& y_{1,7}&  y_{1,7}+3y_{7,2}\\
 0& y_{1,1}& 0& 0& 0& 0& 0& 0\\
 0& 0& y_{1,1}& 0& 0& 0& 0& 0\\
0& 0& 0& y_{1,1}&  0& 2y_{1,7}+3y_{7,2}& 0& 0\\
 0& 0& y_{1,7}+3y_{7,2}& 0& y_{1,1}&  0& 0& 0\\
0& 0& 0& 0& 0& y_{1,1}& 0& 0\\
0& y_{7,2}& 0& 0& 0& 0& y_{1,1}& 0\\
0& y_{1,7}+2y_{7,2}&0& 0& 0& 0& 0& y_{1,1}
\end{pmatrix}.
$$

Besides that we have the condition
$$
y x_{\alpha_2}(1)-w_{\alpha_2}(1) y w_{\alpha_2}(1)^{-1} x_{\alpha_2}(1) y,
$$
which gives, at first, $y_{1,1}(1-y_{1,1})=0$. Since $\det y = y_{1,1}^8$, the $y_{1,1}$ is invertible, so
 $y_{1,1}=1$. Also this condition gives $z_{1,2}=-z_{1,7}^2/4$, $z_{7,2}=-z_{1,7}/2$.

Hence for the root system $A_2$ the assumption is proved.

\medskip

Let us now deal with the root system $B_2$.
Recall that in this system there are roots
 $\pm \alpha_1=\pm(e_1-e_2), \pm \alpha_2=\pm e_2,
 \pm \alpha_3=\pm (\alpha_1+\alpha_2)=\pm e_1, \pm \alpha_4=\pm (\alpha_1+2\alpha_2)=\pm (e_1+e_2)$
 (detailed matrices for this system can be found in the papers \cite{korni2} and \cite{BunBl}).

Consider now $\rho (x_{\alpha_2}(t))=y$ (it is $10\times 10$ matrix).
Note that $y$ commutes with $x_{\alpha_2}(1)$, $x_{-\alpha_1}(1)$, $x_{\alpha_4}(1)$, and also with $w_{\alpha_3}(1)$.
Using this condition we obtain directly (we have not use yet the whole conditions), that
$$
y=\begin{pmatrix}
y_{1,1}& 0& 0& 0& 0& 0& 0& 0& 0& 0\\
 y_{2,1}& y_{1,1}&  0& y_{2,4}& 0& y_{2,6}&  0 & y_{2,8}&  y_{2,9} & y_{2,10}\\
y_{3,1}&  0& y_{1,1}& y_{3,4}& y_{3,5}& y_{3,6}&  0& y_{3,8}& y_{3,9} & y_{3,10}\\
0& 0& 0& y_{1,1}& 0& 0& 0& 0& 0& 0\\
 y_{5,1}& 0& 0& y_{5,4}& y_{1,1}& 0& 0& 0& 0& 0\\
0& 0& 0& y_{6,4}& 0& y_{1,1}&  0& y_{6,8}& 0& 0\\
y_{7,1}& 0& 0& y_{7,4}& y_{7,5}& 0 y_{1,1}& y_{7,8}& y_{7,9}& y_{7,10}\\
0& 0& 0& 0& 0& 0& 0& y_{1,1}& 0& 0\\
y_{9,1}& 0& 0& 0& 0& 0& 0& y_{9,8}& y_{1,1}& 0\\
0& 0& 0& y_{10,4}& 0& 0& 0& y_{10,8}& 0& y_{1,1}
\end{pmatrix}.
$$

Again the determinant of~$y$ is $y_{1,1}^{10}=1$, therefore $y_{1,1}$ is invertible.

Use the condition $h_{\alpha_1}(-1) y h_{\alpha_1}(-1) y=E$, which implies
 $y_{2,1}=y_{2,9}=y_{2,10}=y_{3,5}=y_{3,6}=y_{5,4}=y_{6,4}=y_{9,8}=y_{9,1}=y_{10,8}=y_{7,8}=y_{7,9}=y_{7,10}=0$.
 Besides, $y_{1,1}^2=1$.

Let now $\rho(x_{\alpha_1}(t))=z$. The  matrix $z$ commutes with
$$
h_{\alpha_1}(-1)=\diag[1,1,-1,-1,-1,-1,1,1,1,1],
 $$
 therefore it is block-diagonal up to the basis parts
  $\pm \alpha_1, \pm \alpha_4, h_1,h_2$ and $\pm \alpha_2, \pm \alpha_3$.
  Now we use the fact that $z$ commutes with
   $x_{\alpha_1}(1)$, $x_{\alpha_4}(1)$, $x_{-\alpha_4}(1)$, $x_{\alpha_3}(1)$, after that we directly obtain that $z$ has the form
$$
\begin{pmatrix}
z_{1,1}& z_{1,2}& 0& 0& 0& 0& 0& 0& -2z_{1,10}& z_{1,10}\\
 0& z_{1,1}&  0& 0& 0& 0&  0 & 0&  0 & 0\\
0&  0& z_{1,1}& 0& 0& 0&  0& 0& 0& 0\\
0& 0& 0& z_{1,1}& 0& z_{1,10}& 0& 0& 0& 0\\
 0& 0& -z_{1,10}& 0& z_{1,1}& 0& 0& 0& 0& 0\\
0& 0& 0& 0& 0& z_{1,1}&  0& 0& 0& 0\\
0& 0& 0& 0& 0& 0& z_{1,1}& 0& 0& 0\\
0& 0& 0& 0& 0& 0& 0& z_{1,1}& 0& 0\\
0& z_{1,10}& 0& 0& 0& 0& 0& 0& z_{1,1}& 0\\
0& 0& 0& 0& 0& 0& 0& 0& 0& z_{1,1}
\end{pmatrix}.
$$
Therefore, as above, the element $z_{1,1}$ is invertible.

The matrices $y$ and $z$ by the conditions
$$
y\cdot x_{\alpha_3}(1)= w_{\alpha_2}(1) z^2 w_{\alpha_2}(1)^{-1} x_{\alpha_3}(1)\cdot y
$$
and
$$
z\cdot x_{\alpha_2}(1)= w_{\alpha_1}(1) y w_{\alpha_1}(1)^{-1} w_{\alpha_2}(1) z^2 w_{\alpha_2}(1)^{-1} x_{\alpha_2}(1)\cdot z.
$$
From these conditions, taking into account  invertibility of $z_{1,1}$ and the condition $y_{1,1}^2=1$, we get
 $y=x_{\alpha_2}(z_{1,10})$, $z=x_{\alpha_1}(z_{1,10})$.

The case $B_2$ is studied.

\medskip

The next root system under consideration is $G_2$, also recall that we suppose $2,3\in R^*$.

In the root system $G_2$ there are $12$ roots:
 $\pm \alpha_1, \pm\alpha_2, \pm \alpha_3=\pm (\alpha_1+\alpha_2),
 \pm \alpha_4=\pm (2\alpha_1+\alpha_2), \pm \alpha_5=\pm (3\alpha_1+\alpha_2), \pm \alpha_6=\pm (3\alpha_1+2\alpha_2)$
 (detailed matrices for this system can be found in the paper \cite{korni2}).

For the beginning we consider $y=\rho(x_{\alpha_2}(t))$. Directly from the fact that $y$ commutes with
 $h_{\alpha_2}(-1)$, $x_{\alpha_2}(1)$, $x_{-\alpha_1}(1)$, $x_{\alpha_4}(1)$, $x_{-\alpha_4}(1)$, $x_{\alpha_6}(1)$,
 we obtain $y=y_1 E+y_2 X_{\alpha_2} + y_3 X_{\alpha_2}^2$.

The condition $h_{\alpha_1}(-1) y h_{\alpha_1}(-1) y=E$ gives $y_1^2=1$ and $y_1 y_3=-y_2^2$.

The condition $y\cdot x_{\alpha_5}(1)-w_{\alpha_1}(1) y w_{\alpha_1}(1)^{-1} x_{\alpha_5}(1)\cdot y$ gives $y_1^2=y_1$,
so $y_1=1$ and after that $y_3=-y_2^2$. Consequently, $y=x_{\alpha_2}(y_3)$.

Let now $z=\rho(x_{\alpha_1}(t))$. Again from commuting with $h_{\alpha_1}(-1)$,
$x_{\alpha_1}(1)$, $x_{-\alpha_2}(1)$, $x_{\alpha_5}(1)$, $x_{\alpha_6}(1)$, $x_{-\alpha_6}(1)$ we directly  obtain that $z$  has the form
$$
\begin{pmatrix}
z_1& z_2& 0& 0& 2z_3& -3z_3\\
0& z_1& 0& 0& 0& 0\\
0& 0& z_1& 0& 0& 0\\
0& 0& 0& z_1& 0& 0\\
0& z_3& 0& 0& z_1& 0\\
0& 0& 0& 0& 0& z_1
\end{pmatrix}
$$
on the basis part $\{\pm \alpha_1, \pm \alpha_6, h_1,h_2\}$ and the form
$$
\begin{pmatrix}
z_1& 0& 0& 0& 0& 0& 0& 0\\
0& z_1& 0& z_3& 0& z_4& 0& z_5\\
-3z_3& 0& z_1& 0& 0& 0& 0& 0\\
0& 0& 0& z_1& 0& -2z_3& 0& -3z_4\\
3z_4& 0& 2z_3& 0& z_1& 0& 0& 0\\
0& 0& 0& 0& 0& z_1& 0& -3z_3\\
-z_5& 0& -z_4& 0& z_3& 0& z_1& 0\\
0& 0& 0& 0& 0& 0& 0& z_1
\end{pmatrix}
$$
on the basis part $\{ \pm \alpha_2, \pm \alpha_3, \pm \alpha_4, \pm \alpha_5\}$.

From the condition
$
h_{\alpha_2}(-1) z h_{\alpha_2}(-1) z=E$ we get $z_1=1$, $z_2=-z_3^2$, $z_4=-z_3^2$.

Now use the last condition:
$$
w_{\alpha_2}(1) z w_{\alpha_2}(1)^{-1} x_{\alpha_4}(1)=w_{\alpha_2}(1)w_{\alpha_1}(1)
y^3 w_{\alpha_1}(1)^{-1} w_{\alpha_2}(1)^{-1} x_{\alpha_4}(1) w_{\alpha_2}(1) z w_{\alpha_2}(1)^{-1},
 $$
 which connects $y$ and $z$. From this condition we obtain $z_3=-y_3$. The assumption for $z_5$ follows from the fact that
  $z$ is an  element of the Chevalley group.

So we have studied also the case $G_2$.

\medskip

Let now  $1/2\in R$, the root system be one of $B_l, C_l, F_4$ (matrices corresponding to these root systems can be found in
the papers \cite{BunBl} and \cite{bunF4}).

Let $y$ be  the image of some long simple root (for example, $y=\rho(x_{\alpha_i}(t))$), $z$ be the image of a short root
 ($z=\rho(x_{\alpha_j}(t))$). We can assume that for the system $B_l$ $\alpha_i=\alpha_1$, $\alpha_j=\alpha_l$, for the system
  $C_l$ $\alpha_i=\alpha_l$, $\alpha_j=\alpha_1$, for the system $F_4$ $\alpha_i=\alpha_1$, $\alpha_j=\alpha_4$.

Note that $y$ and $z$ commute with $h_{\alpha}(-1)$ for some definite~$\alpha\in \Phi$, therefore according to invertible $2$
we directly obtain that the matrices $y$ и $z$ are block-diagonal up to some basis separation.

For a long root~$\alpha$  in any system under consideration  all other pairs of roots $\pm \beta$ are divided to the following cases:

1. $\pm \beta$ are also long roots, orthogonal to $\alpha$.

2. $\pm \beta$ are long roots, generating with $\pm \alpha$ the system $A_2$.

3. $\pm \beta$  are short roots, orthogonal to $\alpha$.

4. $\pm \beta$ a short roots, generating with $\alpha$ the system $B_2$.

In the first case the matrix $y$ commutes with $x_{\beta}(1)$ and $x_{-\beta_1}(1)$, therefore $y$ commutes with
$E_{-\beta,\beta}$, $E_{\beta,-\beta}$. It means that on the basis part $-\beta,\beta$ the matrix $y$ is invariant and scalar, and also
the rest basis part  is also invariant. The same thing is with the third case.

In the second case according to commuting  $y$ with $h_{\alpha}(-1)$ and other $h_{\gamma}(-1)$ for  long roots, distinct
from~$\pm \beta$, the basis part $\pm \beta, \pm (\alpha+\beta)$ is separated to the own diagonal block.

The fourth case means that there are roots $\pm \alpha$, $\pm \beta$, $\pm (\alpha+\beta)$ (short) and
 $\pm (\alpha+2\beta)$ (long). Also the long roots $\pm (\alpha +2\beta)$ are orthogonal to $\alpha$ and, as it was shown above,
 the matrix $y$ is scalar on them. According to commuting with $h_{\alpha}(-1)$ the basis part
  $\pm \alpha, h_1,\dots, h_l$ is separated of the basis part $\pm \beta, \pm (\alpha+\beta)$.
  Now we just need to show that the basis part $\pm \beta, \pm (\alpha+\beta)$ is separated of the basis part
   $\pm \gamma, \pm (\alpha+\gamma)$, where $\gamma$ is also a root of the fourth type.
   If the root system under consideration is $B_l$, then $\alpha$ can be  supposed as the root
    $\alpha_1=e_1-e_2$, then $\beta$ can be only $e_2$. If the   root system is $C_l$, then, for example, $\alpha=2e_1$, then
     $\beta=e_i-e_1$, $\gamma=e_j-e_1$, $i,j> 1$. In this case it is clear that the corresponding parts are separated according to
     the fact that $y$ commutes with $h_{e_i}(-1)$ and $h_{e_j}(-1)$. A similar situation is for the root system~$F_4$.

We see that the whole matrix $y$ is divided into diagonal blocks, where every block is either scalar (and corresponds to some
pair of roots $\pm \beta$), or corresponds to the roots $\pm \beta, \pm (\alpha+\beta)$. Now we can use the results, connected with
the root systems $A_2$ and $B_2$, therefore $y=x_{\alpha_i}(s)$ for some $s\in R$.

Let now $\alpha$ be a short root. Then all other roots are divided into the following cases:

1. $\beta$ is a long root, orthogonal to $\alpha$.

2. $\beta$ is a long root, generating with $\alpha$ the root system $B_2$.

3. $\beta$ is a short root, orthogonal to $\alpha$ and generating with it the system $B_2$ (for example, in the root system
 $B_l$ it holds for $\alpha=e_i$, $\beta=e_j$).

4. $\beta$ is a short root, orthogonal to $\alpha$, $\alpha\pm \beta\notin \Phi$ (for example, for the rot system $C_l$ it holds for
 $\alpha=e_1-e_2$, $\beta=e_3-e_4$).

5. $\beta$ is a short root, generating with $\alpha$ the root system $A_2$ (it holds, for example, in the root system $C_l$ for
 $\alpha=e_1-e_2$, $\beta=e_2-e_3$).

Then all considerations are similar to the long roots.

Therefore the both matrices $y$ and $z$ are block-diagonal, the blocks correspond to separating into root systems $A_2$, $B_2$,
 scalar  $2\times 2$ matrices, and also the basis part $\pm \alpha, h_1,\dots, h_l$ is separated. On the basis part
  $h_3,\dots, h_l$ the matrices are scalar since they commutes with corresponding
   $w_{\alpha_3}(1)$, \dots, $w_{\alpha_l}(1)$. Since we have studied the cases  $A_2$ and $B_2$ above, we see that on the basis parts
    $\pm \alpha, \pm \beta, \pm (\alpha+\beta), h_1, h_2$, and also on the parts
     $\pm \alpha, \pm \gamma, \pm (\alpha+\gamma), \pm (2\alpha  +  \gamma)$ (or $\pm (\alpha+2\gamma)$ depending on the length of~$\alpha$),
     $h_1,h_2$ the matrices coincide with $x_\alpha (s)$ for some $s\in R$. Since on the basis part $\pm \alpha, h_1,h_2$
     we always have the same $s$, then  $s$ is unique for all other basis parts. On the places where the matrix is scalar,
     the multiplier is the same according to commuting with Weil group elements, which map the roots, orthogonal to~$\alpha$,
     into each other, and which preserve~$\alpha$. Clear that either $y$ (in the root system $B_l$, for example), or $z$
     (in the root system $C_l$), and in the root system $F_4$ both $y$ and $z$ can be embedded into the root system $A_2$,
     where we have the following condition (here we write it for~$y$)
$$
y\cdot x_{\gamma}(1)= w y w^{-1} x_{\gamma}(1) \cdot y.
$$
If the matrix $y$ on it scalar part has the multiplier $a$, then this condition gives us either $a=1$, or $a=a^2$,
and according to invertibility of~$a$ it also implies $a=1$. Therefore either $y$, or $z$ coincides with $x_\alpha (s)$ for some $s\in R$.

Now we can assume (according to conjugations by the Weil group elements), that $y$ and $z$ are images of $x_\alpha(t)$ and $x_\beta(t)$,
where $\alpha$ and $\beta$ generate the root system $B_2$. Clear that after using the corresponding commutator conditions we obtain that
all indefinite scalars are~$1$. Therefore, $y=x_{\alpha}(s_1)$, $z=x_{\beta}(s_2)$. Also it is clear (from the above case
 $B_2$), that $s_1=s_2$.

Consequently, the lemma is proved for the root systems $A_2, B_l, C_l, F_4, G_2$.

\medskip

Now we need to prove the lemma for the root systems $A_l, D_l, E_l$, $l\geqslant 3$, but without the condition $1/2\in R$.

At first we are going to show that the lemma holds for the system $A_3$.

In the root system $A_3$ there are roots $\pm \alpha_1, \pm \alpha_2, \pm \alpha_3,
\pm \alpha_4=\pm (\alpha_1+\alpha_2), \pm \alpha_5=\pm (\alpha_2+\alpha_3), \pm \alpha_6=\pm (\alpha_1+\alpha_2+\alpha_3)$
(detailed matrices for this system can be found in the paper~\cite{without2}).

Let $y=\rho(x_{\alpha_1}(t))$.

Note that there is the condition
 $(x_\alpha(1)x_\beta(1)-x_\alpha(1)-x_\beta(1)+E)^2=E_{\alpha+\beta,-\alpha-\beta}$, if $\alpha+\beta\in \Phi$.

Since $y$ commutes with $x_{\alpha_1}(1)$, $x_{\alpha_1+\alpha_2}(1)$, $x_{\alpha_3}(1)$, $x_{-\alpha_2}(1)$,
$x_{-\alpha_3}(1)$, then $y$ has to commute also with $E_{\alpha_6,-\alpha_6}$, $E_{-\alpha_5,\alpha_5}$, and with
 $E_{\alpha_2,-\alpha_2}$. Besides, $y$ commutes with $w_{\alpha_3}(1)$, therefore we obtain that the lines of~$y$,
 corresponded to the basis vectors $\alpha_6, -\alpha_5, \alpha_2$, and its rows corresponded to the vectors
  $-\alpha_6, \alpha_5, -\alpha_2$, have nonzero elements only on the diagonal.

Then we can directly use commuting with matrices written above, commutator condition and the fact that $y$ belongs to
the corresponding Chevalley group and obtain that $y=x_{\alpha_1}(s)$, what was required.

 \medskip

Suppose, finally, that we deal with an arbitrary root system under consideration, still we set $y=\rho(x_{\alpha_1}(t))$.

All basis elements are divided up to  $\alpha_1$ to the following cases:

1. $\pm \alpha_1$ themselves, and also $h_1, h_2$.

2. $h_3,\dots, h_l$.

3. $\pm \beta$, where a root  $\beta$ is orthogonal to~$\alpha_1$, and also there exists one more root $\gamma$, orthogonal to
 $\alpha_1$ and not orthogonal to $\pm \beta$ (it  always holds for the root system $A_l$, $l>3$, but, for example, for the root system
  $D_l$ it does not hold for $\alpha_1=e_1-e_2$, $\beta=e_1+e_2$).

4. $\pm \beta$, where the root $\beta$ is orthogonal to~$\alpha_1$  and does not satisfy the assertion~3.

5. $\pm \beta$, where  $\alpha_1$ and $\beta$ generate the root system~$A_2$.

To use the result, obtained for the system~$A_3$, we just need to prove that the matrix
 $y$ is block-diagonal, where every block has one of listed above types or the block
 of the root system~$A_3$.

At first we consider the easy case --- type 3. For this case we take a root $\gamma$, orthogonal to~$\alpha$ and such that
 $\beta+\gamma\in \Phi$ (clear that $\beta+\gamma$ is also orthogonal to~$\alpha_1$). In the same manner as it was
 described in consideration of the case $A_3$, we obtain that  $y$ commutes with $E_{-\beta,\beta}$ and $E_{\beta,-\beta}$,
 therefore we directly have that the block $\pm \beta$ is separated in the matrix $y$, on this block $y$ is scalar.
 Using the commutator relation we obtain that $y$ on this block is identical. Consequently on the basis parts of the third type
  $y$ completely coincides with $x_{\alpha_1}(s)$.

Besides, $(x_{\alpha_i}(1)-E+E_{\alpha_i,-\alpha_i})E_{-\alpha_i,-\alpha_i}=E_{h_i,-\alpha_i}$ commutes with $y$ for $i\geqslant 3$.
So we directly obtain that the whole row of $y$ with number $h_i$ is zero, except a diagonal element.
From another side,
$E_{\alpha_i,\alpha_i}(x_{\alpha_i}(1)-E+E_{\alpha_i,-\alpha_i})=E_{\alpha_i,h_{i-1}}-2E_{\alpha_i,h_i}+E_{\alpha_i,h_{i+1}}$
also commutes with  $y$, therefore we have that between $h_{i-1}$-th, $h_i$-th and $h_{i+1}$-th lines of $y$ there exists a natural
dependence,
i.\,e., if we show that the $h_2$-th line is zero, then other lines become zero  (except diagonal elements which will be equal).

According to this fact and the considered case $A_3$ we just need to study roots of forth and fifth types.

At the beginning we heed to show that if  roots $\beta$ and $\gamma$ have one of these types and are orthogonal to each other,
then on the place $\beta,\gamma$ in the matrix~$y$ there is zero.

Let the both roots  $\beta$ and $\gamma$ have the forth type. Clear that in this case   they  together with $\alpha_1$
are embedded in the rot system $D_4$, i.\,e., for our convenience we can assume that
 $\alpha_1=e_1-e_2$, $\beta=e_3-e_4$, $\gamma=e_3+e_4$. The matrix $y$ commutes with
  $x_{e_1\pm e_3}(1)$, $x_{e_1 \pm e_4}(1)$, $x_{-e_2\pm e_3}(1)$, $x_{-e_2\pm e_4}(1)$, $x_{\pm e_3\pm e_4}(1)$,
  therefore, as before, $y$ commutes with
   $E_{e_1\pm e_3, -e_1\mp e_3}$, $E_{e_1\pm e_4, -e_1\mp e_4}$, $E_{-e_2\pm e_3, e_2\mp e_3}$, $E_{-e_2\pm e_4, e_2\mp e_4}$.

The matrix $_{e_1-e_3}(1)-E$ at the line corresponding to the root $e_3-e_4$, has only one nonzero element $E_{e_3-e_4,e_1-e_4}$, so
$E_{e_3-e_4, -e_1+e_4}=(x_{e_1-e_3}(1)-E)E_{e_1-e_4,-e_1+e_4}$ commutes with~$y$.

If we multiply $E_{e_3-e_4,-e_1+e_4}$ to $x_{e_1+e_3}(1)-E$, we get $E_{e_3-e_4,e_3+e_4}$, the matrix $y$ also commutes with it.
It is sufficient to have zero on the place $e_3-e_4, e_3+e_4$ in the matrix $y$, what was obtained.

Note that for the root systems $A_l$ and $E_l$ roots of the fourth type do not exist. Let us consider the root system $D_l$,
the root $\beta$ has the fourth type, the root $\gamma$ has the fifth type, the roots  are orthogonal to each other.
It can not be in the system $D_4$, so we can assume that $\alpha_1=e_1-e_2$, $\beta=e_1+e_2$. But in this case every root
which is orthogonal to $\beta$, is also orthogonal to~$\alpha_1$. The situation is impossible.

Now let both roots $\beta$ and $\gamma$ have the fifth type up to~$\alpha_1$, and are orthogonal to each other. Clear that
they can be embedded into the system $A_3$, i.\,e., we can assume that $\alpha_1=e_1-e_2$, $\beta=e_3-e_1$, $\gamma=e_2-e_4$.
In this case zeros on the corresponding places evidently follow from the consideration  of the system~$A_3$.

We have completely studied the case with orthogonal roots $\beta$ and $\gamma$.

Now suppose that roots $\beta$ and $\gamma$ are not orthogonal to each other, i.\,e., generate the system $A_2$.
Clear that in this case they cannot both be of the fourth type. Let the root $\beta$ be of the fourth type, the root $\gamma$
be of the fifth type.
It means that the roots $\alpha_1$, $\beta$, $\gamma$ together generate the system $A_3$, we can assume that
 $\alpha_1=e_1-e_2$, $\gamma=e_2-e_3$, $\beta=e_3-e_4$. This case was already considered, according to commuting with the
 corresponding
$x_\alpha(1)$ we have already proved that there is zero on the place $\beta, \gamma$ in the matrix~$y$.

Finally, let $\beta$ and $\gamma$ both belong to the fifth type. Then we can assume that
 $\alpha_1=e_1-e_2$, $\beta=e_2-e_3$, $\gamma=e_4-e_1$.
 It is clear again that this case is considered for the root system $A_3$.

Therefore we have shown that the matrix $y$ is divided into diagonal blocks so that we can apply the results for $A_3$ to every block.

Consequently, $y=x_{\alpha_1}(s)$, what was required.

\medskip

Consequently for all cases  under consideration  we obtain that $\rho(x_\alpha(t))=x_\alpha(s)$, the mapping
 $t\mapsto s$ does not depend of choice of~$\alpha\in \Phi$. Denote this mapping also by~$\rho: R\to R$,
 we only need to prove that it is an automorphism of~$R$.

Actually, it is one-to-one because the initial automorphism $\rho\in \Aut (E_{\ad}(\Phi, R))$ is bijective.

Its additivity follows from the formula
\begin{multline*}
x_{\alpha}(\rho(t_1)+\rho(t_2))=x_\alpha(\rho(t_1))x_\alpha(\rho(t_2))=\rho(x_\alpha(t_1))\cdot \rho   (x_\alpha(t_2))=\\
=\rho(x_\alpha(t_1)x_\alpha(t_2))=\rho(x_\alpha(t_1+t_2)=x_\alpha (\rho(t_1+t_2)),
\end{multline*}
and  its multiplicativity follows from
\begin{multline*}
x_{\alpha_1+\alpha_2}(\rho(t_1)\rho(t_2))=[x_{\alpha_1}(\rho(t_1)), x_{\alpha_2}(\rho(t_2))]=\rho([x_{\alpha_1}(t_1),
x_{\alpha_2}(t_2)])= \\
=\rho(x_{\alpha_1+\alpha_2}(t_1t_2))=x_{\alpha_1+\alpha_2}(\rho(t_1t_2))
\end{multline*}
for roots $\alpha_1$ and $\alpha_2$, forming the system $A_2$. We cannot find such  a pair of roots only in the system~$B_2$, where we can
take the short roots $\alpha_1=e_1$, $\alpha_2=e_2$, the formula will be
$$
[x_{\alpha_1}(t_1),x_{\alpha_2}(t_2)]=x_{\alpha_1+\alpha_2}(2t_1t_2),
$$
therefore
$$
2\rho (t_1t_2)=2\rho(t_1)\rho(t_2),
$$
it is sufficient because for the root systems $B_l$ we suppose $2$ to be invertible.

Lemma is proved.
\end{proof}

Now we have that the initial automorphism $\varphi$ of the elementary Chevalley group $E_{\ad} (\Phi, R)$ is the composition
of the conjugation $\psi$ with some matrix $A\in \GL_N (S)$ and a ring automorphism~$\rho$, also we know that the conjugation
 $\psi$ is an automorphism of the Lie algebra ${\mathcal L}(\Phi, R)$.

Now  we use the description of such automorphisms from the paper~\cite{Klyachko}  (see Theorem~1 of this paper):

\begin{lemma}\label{Klyachko}
Let $R$ be a commutative associative ring, $\Phi$ be an undecomposable root system. Then every automorphism $\psi$ of the Lie algebra
${\mathcal L}(\Phi, R)$ is the composition of graph and inner automorphisms (by inner automorphism we mean a conjugation with
some $g\in G_{\ad}(\Phi,R)$).
\end{lemma}

\begin{proof}
Consider the ideal  $J$ in $\mathbb Z [x_{1,1}, x_{1,2}, \dots , x_{N,N}]$, defining the group
 $\Aut_{\mathbb C} {\mathcal L}(\Phi)$. Over complex numbers this group is well-known: the ideal $J$ decomposes into a product
  $J = J_1J_2 \dots J_d$ of prime ideals~$J_i$,
 corresponding to irreducible
($=$ connected) components $ h_iG(\Phi)$ of the group $\Aut_{\mathbb C} {\mathcal L}(\Phi)$, where $h_i$
 are integer matrices of diagram automorphisms.

 Take a matrix $A = (a_{p,q}) \in \Aut_R{\mathcal L}(\Phi,R)$. Then  $f(a_{p,q}) = 0$ for $f \in J$.
 Put $I_i = \{ f(a_{p,q}) \mid f \in  J_i\} \triangleleft R$.
Then

(i) $\prod I_i=0\{ 0\}$;

(ii) $I_i + I_j = R$ for $i\ne j$ (otherwise we take the factor ring by a maximal ideal $M \supset I_i + I_j$ and obtain a matrix
 $A_M$, belonging to the intersection of two irreducible components of the group $\Aut_{R/M}{\mathcal L}(\Phi,R/M)$,
 but this intersection
is empty, because  $R/M$ is a field).

These  conditions (i) and (ii) imply that the ring $R$ is the direct sum $R = \oplus R/I_i$ (see~\cite{Bou61}, Chapter~2 \S\,1, Proposition~5).

So, $A =\sum A_{I_i}$, and the entries of the matrix $A_{I_i}\in M_N(R/I_i)$ satisfy the equations $f(a_{p,q}) = 0$ for $f \in I_i$.
Therefore,
$A_{I_i} = h_ig_i \in  h_i G_{\ad}(\Phi,R/I_i)$, and $A = (\sum e_ih_i)(\sum g_i)$, where $e_i$ is the unity of the ring $R/I_i$.

Lemma is proved.
\end{proof}

Consequently, the step 2 is completely proved for adjoint elementary Chevalley groups, i.\,e., every automorphism of such
a group is the composition of inner, graph and ring automorphisms.

The second step is complete, i.e., Theorem~1 is proved.

\end{document}